\documentclass[reqno]{amsart}
\usepackage{amsmath, amsthm, amssymb, amstext}

\usepackage{hyperref,xcolor}
\hypersetup{
 pdfborder={0 0 0},
 colorlinks,
}

\usepackage{enumitem}
\newtheorem{theorem}{Theorem}
\newtheorem{remark}[theorem]{Remark}

\newtheorem{proposition}[theorem]{Proposition}
\newtheorem{corollary}[theorem]{Corollary}

\DeclareMathOperator*{\divergenz}{div}              %
\DeclareMathOperator*{\dist}{dist}          %
\DeclareMathOperator*{\ints}{int}         %
\DeclareMathOperator*{\ww}{w}         %
\DeclareMathOperator*{\Ss}{S}         %

\newcommand{\N}{\mathbb{N}}

\newcommand{\R}{\mathbb{R}}

\newcommand{\Lp}[1]{L^{#1}(\Omega)}

\newcommand{\Wp}[1]{W^{1,#1}(\Omega)}

\newcommand{\Wpzero}[1]{W^{1,#1}_0(\Omega)}

\newcommand{\lan}{\langle}
\newcommand{\ran}{\rangle}
\newcommand{\eps}{\varepsilon}
\newcommand{\ph}{\varphi}

\newcommand{\into}{\int_{\Omega}}
\newcommand{\weak}{\overset{\ww}{\to}}

\newcommand{\Linf}{L^{\infty}(\Omega)}
\newcommand{\close}{\overline{\Omega}}
\newcommand{\interior}{\ints \left(C^1_0(\overline{\Omega})_+\right)}

\newcommand{\cprime}{$'$}

\renewcommand{\l}{\left}
\renewcommand{\r}{\right}

\numberwithin{theorem}{section}
\numberwithin{equation}{section}

\title[On a class of singular anisotropic $(p,q)$-equations]{On a class of singular anisotropic $(p,q)$-equations}

\author[N.\,S.\,Papageorgiou]{Nikolaos S.\,Papageorgiou}
\address[N.\,S.\,Papageorgiou]{National Technical University, Department of Mathematics, Zografou Campus, Athens 15780, Greece}
\email{npapg@math.ntua.gr}

\author[P.\,Winkert]{Patrick Winkert}
\address[P.\,Winkert]{Technische Universit\"{a}t Berlin, Institut f\"{u}r Mathematik, Stra\ss e des 17.\,Juni 136, 10623 Berlin, Germany}
\email{winkert@math.tu-berlin.de}


\begin{document}

\begin{abstract}
	We consider a Dirichlet problem driven by the anisotropic $(p,q)$-Laplacian and with a reaction that has the competing effects of a singular term and of a parametric superlinear perturbation. Based on variational tools along with truncation and comparison techniques, we prove a bifurcation-type result describing the changes in the set of positive solutions as the parameter varies.
\end{abstract}

\maketitle

\section{Introduction}

Let $\Omega \subseteq \R^N$ be a bounded domain with a $C^2$-boundary $\partial \Omega$. In this paper, we study the following anisotropic Dirichlet problem
\begin{align}\tag{P$_\lambda$}\label{problem}
  \begin{split}
    &-\Delta_{p(\cdot)}u-\Delta_{q(\cdot)} u = u^{-\eta(x)}+\lambda f(x,u)\quad \text{in } \Omega \\
    &u\big|_{\partial \Omega}=0, \quad u>0, \quad \lambda>0.
   \end{split}
\end{align}
For $r \in E_1$, where $E_1$ is given by 
\begin{align*}
E_1=\l \{r \in C(\close)\, : \, 1<\min_{x\in \close} r(x) \r\},
\end{align*}
we denote by $\Delta_{r(\cdot)}$ the anisotropic $r$-Laplacian (or $r(\cdot)$-Laplacian) defined by
\begin{align*}
        \Delta_{r(\cdot)} u = \divergenz \l(|\nabla u|^{r(x)-2} \nabla u \r) \quad\text{for all }u \in \Wpzero{r(\cdot)}.
\end{align*}
The differential operator in problem \eqref{problem} is the sum of two such operators. In the reaction, the right-hand side of \eqref{problem}, we have the competing effects of two terms which are of different nature. One is the singular term $s \to s^{-\eta(x)}$ for $s>0$ with $\eta \in C(\close)$ such that $0<\eta(x)<1$ for all $x\in \close$.  The other one is the parametric term $s \to \lambda f(x,s)$ with $\lambda>0$ being the parameter and $f\colon \Omega\times\R\to\R$ is a Carath\'eodory function, that is, $x\to f(x,s)$ is measurable for all $s\in\R$ and $s\to f(x,s)$ is continuous for a.\,a.\,$x\in\Omega$. We assume that $f(x,\cdot)$ exhibits $(p_+-1)$-superlinear growth for a.\,a.\,$x\in\Omega$ as $s\to +\infty$ with $p_+=\max_{x\in\close}p(x)$. We are looking for positive solutions of problem \eqref{problem} and our aim is to determine how the set of positive solutions of \eqref{problem} changes as the parameter $\lambda$ moves on the semiaxis $\overset{\circ}{\R}_+=(0,+\infty)$.

The starting point of our work is the recent paper of Papageorgiou-Winkert \cite{15-Papageorgiou-Winkert-2019} where the authors study a similar problem driven by the isotropic $p$-Laplacian. So, the differential operator in \cite{15-Papageorgiou-Winkert-2019} is $(p-1)$-homogeneous and this property is exploited in their arguments. In contrast here, the differential operator is both nonhomogeneous and anisotropic.

Anisotropic problems with competition phenomena in the source were recently investigated by Papageorgiou-R\u{a}dulescu-Repov\v{s} \cite{12-Papageorgiou-Radulescu-Repovs-2020}. They studied concave-convex problems driven by the $p(\cdot)$-Laplacian plus an indefinite potential term. In their equation there is no singular term. In fact, the study of anisotropic singular problems is lagging behind. We are aware only the works of Byun-Ko \cite{2-Byun-Ko-2017} and Saoudi-Ghanmi \cite{20-Saoudi-Ghanmi-2017} for Dirichlet as well as of  Saoudi-Kratou-Alsadhan \cite{21-Saoudi-Kratou-Alsadhan-2016} for Neumann problems. All the aforementioned works deal with equations driven by the $p(\cdot)$-Laplacian.

We mention that equations driven by the sum of two differential operators of different nature arise often in the mathematical models of physical processes. We mention the works of Bahrouni-R\u{a}dulescu-Repov\v{s} \cite{1-Bahrouni-Radulescu-Repovs-2019} (transonic flow problems), Cherfils-Il\cprime yasov \cite{3-Cherfils-Ilyasov-2005} (reaction diffusion systems) and Zhikov \cite{25-Zhikov-2011} (elasticity problems). Some recent regularity and multiplicity results can be found in the works of Ragusa-Tachikawa \cite{19-Ragusa-Tachikawa-2020} and Papageorgiou-Zhang \cite{17-Papageorgiou-Zhang-2020}.

In this paper, under general conditions on the perturbation $f\colon\Omega\times\R\to\R$ which are less restrictive than all the previous cases in the literature, we prove the existence of a critical parameter $\lambda^*>0$ such that
\begin{enumerate}
	\item[$\bullet$]
		for every $\lambda\in(0,\lambda^*)$, problem \eqref{problem} has at least two positive smooth solutions;
	\item[$\bullet$]
		for $\lambda=\lambda^*$, problem \eqref{problem}  has at least one positive smooth solution;
	\item[$\bullet$]
		for every $\lambda>\lambda^*$, problem \eqref{problem} has no positive solutions.
\end{enumerate}

\section{Preliminaries and Hypotheses}

The study of anisotropic equations uses Lebesgue and Sobolev spaces with variable exponents. A comprehensive presentation of the theory of such spaces can be found in the book of Diening-Harjulehto-H\"{a}st\"{o}-R$\mathring{\text{u}}$\v{z}i\v{c}ka \cite{4-Diening-Harjulehto-Hasto-Ruzicka-2011}.

Recall that $E_1=\{r\in C(\close) \, : \, 1<\min_{x\in \close} r(x) \}$. For any $r\in E_1$ we define
\begin{align*}
	r_-=\min_{x\in \close}r(x) \qquad\text{and}\qquad r_+=\max_{x\in\close} r(x).
\end{align*}
Moreover, let $M(\Omega)$ be the space of all measurable functions $u\colon \Omega\to\R$. As usual, we identify two such functions when they differ only on a Lebesgue-null set. Then, given $r \in E_1$, the variable exponent Lebesgue space $\Lp{r(\cdot)}$ is defined as
\begin{align*}
	\Lp{r(\cdot)}=\l\{u \in M(\Omega)\,:\, \into |u|^{r(x)}\,dx<\infty \r\}.
\end{align*}
We equip this space with the so-called Luxemburg norm defined by
\begin{align*}
	\|u\|_{r(\cdot)} =\inf \l \{\lambda>0 \, : \, \into \l(\frac{|u|}{\lambda}\r)^{r(x)}\,dx \leq 1 \r\}.
\end{align*}
Then $(\Lp{r(\cdot)},\|\cdot\|_{r(\cdot)})$ is a separable and reflexive Banach space, in fact it is uniformly convex. Let $r' \in E_1$ be the conjugate variable exponent to $r$, that is,
\begin{align*}
	\frac{1}{r(x)}+\frac{1}{r'(x)}=1 \quad\text{for all }x\in\Omega.
\end{align*}
We know that $\Lp{r(\cdot)}^*=\Lp{r'(\cdot)}$ and the following H\"older type inequality holds
\begin{align*}
	\into |uv| \,dx \leq \l[\frac{1}{r_-}+\frac{1}{r'_-}\r] \|u\|_{r(\cdot)}\|v\|_{r'(\cdot)}
\end{align*}
for all $u\in \Lp{r(\cdot)}$ and for all $v \in \Lp{r'(\cdot)}$.

If $r_1, r_2\in E_1$ and $r_1(x) \leq r_2(x)$ for all $x\in \close$, then we have that
\begin{align*}
	\Lp{r_2(\cdot)} \hookrightarrow \Lp{r_1(\cdot)} \quad \text{continuously}.
\end{align*}

The corresponding variable exponent Sobolev spaces can be defined in a natural way using the variable exponent Lebesgue spaces. So, if $r \in E_1$, then the variable exponent Sobolev space $\Wp{r(\cdot)}$ is defined by
\begin{align*}
	\Wp{r(\cdot)}=\l\{ u \in \Lp{r(\cdot)} \,:\, |\nabla u| \in \Lp{r(\cdot)}\r\}.
\end{align*}

Here the gradient $\nabla u$ is understood in the weak sense. We equip $\Wp{r(\cdot)}$ with the following norm
\begin{align*}
	\|u\|_{1,r(\cdot)}=\|u\|_{r(\cdot)}+\||\nabla u|\|_{r(\cdot)} \quad\text{for all } u \in \Wp{r(\cdot)}.
\end{align*}
In what follows we write $\|\nabla u\|_{r(\cdot)}= \||\nabla u|\|_{r(\cdot)}$. Suppose that $r \in E_1$ is Lipschitz continuous, that is, $r_1 \in E_1 \cap C^{0,1}(\close)$. We define 
\begin{align*}
	\Wpzero{r(\cdot)}= \overline{C^\infty_c(\Omega)}^{\|\cdot\|_{1,r(\cdot)}}.
\end{align*}
The spaces $\Wp{r(\cdot)}$ and $\Wpzero{r(\cdot)}$ are both separable and reflexive, in fact uniformly convex Banach spaces. On the space $\Wpzero{r(\cdot)}$ we have the Poincar\'e inequality, namely there exists $c_0>0$ such that 
\begin{align*}
	\|u\|_{r(\cdot)} \leq c_0 \|\nabla u\|_{r(\cdot)} \quad\text{for all } u \in \Wpzero{r(\cdot)}.
\end{align*}
Therefore, we can consider on $\Wpzero{r(\cdot)}$ the equivalent norm
\begin{align*}
	\|u\|_{1,r(\cdot)}=\|\nabla u\|_{r(\cdot)} \quad\text{for all } u \in \Wpzero{r(\cdot)}.
\end{align*}
For $r \in E_1$ we introduce the critical Sobolev variable exponent $r^*$ defined by
\begin{align*}
	r^*(x)=
	\begin{cases}
		\frac{Nr(x)}{N-r(x)} & \text{if }r(x)<N,\\
		+\infty & \text{if } N \leq r(x),
	\end{cases} \quad\text{for all }x\in\close.
\end{align*}
Suppose that $r \in E_1 \cap C^{0,1}(\close)$, $q \in E_1$, $q_+<N$ and $1 <q(x) \leq r^*(x)$ for all $x\in\close$. Then we have
\begin{align*}
	\Wpzero{r(\cdot)} \hookrightarrow \Lp{q(\cdot)} \quad \text{continuously}.
\end{align*} 
Similarly, if $1 <q(x) < r^*(x)$ for all $x\in\close$, we have
\begin{align*}
\Wpzero{r(\cdot)} \hookrightarrow \Lp{q(\cdot)} \quad \text{compactly}.
\end{align*} 

In the study of the variable exponent spaces, the modular function is important, that is, for $r\in E_1$,
\begin{align*}
	\varrho_{r(\cdot)}(u) =\into |u|^{r(x)}\,dx \quad\text{for all } u\in\Lp{r(\cdot)}.
\end{align*}
As before we write $\varrho_{r(\cdot)}(\nabla u)=\varrho_{r(\cdot)}(|\nabla u|)$. The importance of this function comes from the fact that it is closely related to the norm of the space. This is evident in the next proposition.

\begin{proposition}\label{proposition_1}
	If $r\in E_1$, then we have the following assertions:
	\begin{enumerate}
		\item[(a)]
			$\|u\|_{r(\cdot)}=\lambda \quad\Longleftrightarrow\quad \varrho_{r(\cdot)}\l(\frac{u}{\lambda}\r)=1$ for all $u \in \Lp{r(\cdot)}$ with $u\neq 0$;
		\item[(b)]
			$\|u\|_{r(\cdot)}<1$ (resp. $=1$, $>1$) $\quad\Longleftrightarrow\quad \varrho_{r(\cdot)}(u)<1$ (resp. $=1$, $>1$);
		\item[(c)]
			$\|u\|_{r(\cdot)}<1$ $\quad\Longrightarrow\quad$ $\|u\|_{r(\cdot)}^{r_+} \leq \varrho_{r(\cdot)}(u) \leq \|u\|_{r(\cdot)}^{r_-}$;
		\item[(d)]
			$\|u\|_{r(\cdot)}>1$ $\quad\Longrightarrow\quad$ $\|u\|_{r(\cdot)}^{r_-} \leq \varrho_{r(\cdot)}(u) \leq \|u\|_{r(\cdot)}^{r_+}$;
		\item[(e)]
			$\|u_n\|_{r(\cdot)} \to 0 \quad\Longleftrightarrow\quad\varrho_{r(\cdot)}(u_n)\to 0$;
		\item[(f)]
			$\|u_n\|_{r(\cdot)}\to +\infty \quad\Longleftrightarrow\quad \varrho_{r(\cdot)}(u_n)\to +\infty$.
	\end{enumerate}
\end{proposition}

We know that for $r\in E_1\cap C^{0,1}(\close)$, we have
\begin{align*}
	\Wpzero{r(\cdot)}^*=W^{-1,r'(\cdot)}(\Omega).
\end{align*}
Then we can introduce the nonlinear map $A_{r(\cdot)}\colon \Wpzero{r(\cdot)}\to W^{-1,r'(\cdot)}(\Omega)$ defined by
\begin{align*}
	\l\lan A_{r(\cdot)}(u),h\r\ran= \into |\nabla u|^{r(x)-2} \nabla u \cdot \nabla h\,dx\quad\text{for all }u,h\in\Wpzero{r(\cdot)}.
\end{align*} 
This map has the following properties, see, for example Gasi\'nski-Papageorgiou \cite[Proposition 2.5]{7-Papageorgiou-Gasinski-2011} and R\u{a}dulescu-Repov\v{s} \cite[p.\,40]{18-Radulescu-Repovs-2015}.

\begin{proposition}\label{proposition_2}
	The operator $A_{r(\cdot)}\colon \Wpzero{r(\cdot)}\to W^{-1,r'(\cdot)}(\Omega)$ is bounded (so it maps bounded sets to bounded sets), continuous, strictly monotone (which implies it is also maximal monotone) and of type $\Ss_+$, that is,
	\begin{align*}
		u_n\weak u \text{ in }\Wpzero{r(\cdot)} \text{ and } \limsup_{n\to\infty} \l\lan A_{r(\cdot)}(u_n),u_n-u\r\ran \leq 0
	\end{align*}
	imply $u_n\to u$ in $\Wpzero{r(\cdot)}$.
\end{proposition}

Another space that we will use as a result of the anisotropic regularity theory is the Banach space
\begin{align*}
	C^1_0(\close)=\l\{u\in C^1(\close)\, :\, u\big|_{\partial\Omega}=0  \r\}.
\end{align*}
This is an ordered Banach space with positive (order) cone
\begin{align*}
	C^1_0(\overline{\Omega})_+=\left\{u \in C^1_0(\overline{\Omega}): u(x) \geq 0 \text{ for all } x \in \overline{\Omega}\right\}.
\end{align*}
This cone has a nonempty interior given by
\begin{align*}
\ints \left(C^1_0(\overline{\Omega})_+\right)=\left\{u \in C^1_0(\overline{\Omega})_+: u(x)>0 \text{ for all } x \in \Omega \text{, } \frac{\partial u}{\partial n}\bigg|_{\partial\Omega}<0  \right\},
\end{align*}
where $\frac{\partial u}{\partial n}=\nabla u \cdot n$ with $n$ being the outward unit normal on $\partial \Omega$.

Let $h_1,h_2\in M(\Omega)$. We write $h_1 \preceq h_2$ if and only if $0<c_K\leq h_2(x)- h_1(x)$ for a.\,a.\,$x\in K$ and for all compact sets $K\subseteq \Omega$. It is clear that if $h_1,h_2\in C(\Omega)$ and $h_1(x)< h_2(x)$ for all $x\in \Omega$, then $h_1 \preceq h_2$. From Papageorgiou-R\u{a}dulescu-Repov\v{s} \cite[Proposition 2.4]{12-Papageorgiou-Radulescu-Repovs-2020} and Papageorgiou-R\u{a}dulescu-Repov\v{s} \cite[Propositions 6 and 7]{11-Papageorgiou-Radulescu-Repovs-2020}, we have the following comparison principles. In what follows, let $p,q \in E_1\cap C^{0,1}(\close)$ with $q(x)<p(x)$ for all $x\in \close$ and $\eta \in C(\close)$ with $0 <\eta(x) <1$ for all $x\in \close$.

\begin{proposition}\label{proposition_3}$~$
	\begin{enumerate}
		\item[(a)]
			If $\hat{\xi} \in \Linf$, $\hat{\xi}(x) \geq 0$ for a.\,a.\,$x\in \Omega$, $h_1, h_2\in \Linf$, $h_1\preceq h_2$, $u\in C^1_0(\overline{\Omega})_+$, $u>0$ for all $x\in \Omega$, $v\in\interior$ and
			\begin{align*}
				-\Delta_{p(\cdot)}u-\Delta_{q(\cdot)}u+\hat{\xi} (x) u^{p(x)-1}-u^{-\eta(x)}=h_1(x) & \text{ in }\Omega,\\
				-\Delta_{p(\cdot)}v-\Delta_{q(\cdot)}v+\hat{\xi} (x) v^{p(x)-1}-v^{-\eta(x)}=h_2(x) & \text{ in }\Omega,
			\end{align*}
			then $v-u \in\interior$.
		\item[(b)]
			If $\hat{\xi} \in \Linf$, $\hat{\xi} \geq 0$ for a.\,a.\,$x\in \Omega$, $h_1, h_2\in \Linf$, $0<\hat{c}\leq h_2(x)-h_1(x)$ for a.\,a.\,$x\in\Omega$, $u,v\in C^1(\close)\setminus \{0\}$, $u(x) \leq v(x)$ for all $x\in \Omega$, $v\in\interior$ and
			\begin{align*}
			-\Delta_{p(\cdot)}u-\Delta_{q(\cdot)}u+\hat{\xi} (x) u^{p(x)-1}-u^{-\eta(x)}=h_1(x) & \text{ in }\Omega,\\
			-\Delta_{p(\cdot)}v-\Delta_{q(\cdot)}v+\hat{\xi} (x) v^{p(x)-1}-v^{-\eta(x)}=h_2(x) & \text{ in }\Omega,
			\end{align*}
			then $u(x)<v(x)$ for all $x\in\Omega$.
	\end{enumerate}
\end{proposition}

\begin{remark}
	Note that in part (a) of Proposition \ref{proposition_3} we have by the weak comparison principle that $u \leq v$, see Tolksdorf \cite{23-Tolksdorf-1983}.
\end{remark}

If $u,v\in\Wpzero{p(\cdot)}$ with $u\leq v$, then we define
\begin{align*}
[u,v]&=\left\{y\in\Wpzero{p(\cdot)}: u(x) \leq y(x) \leq v(x) \text{ for a.\,a.\,}x\in\Omega\right\},\\[1ex]
[u) & = \left\{y\in \Wpzero{p(\cdot)}: u(x) \leq y(x) \text{ for a.\,a.\,}x\in\Omega\right\}.
\end{align*}

In what follows we will denote by $\|\cdot\|$ the norm of the Sobolev space $\Wpzero{p(\cdot)}$. By the Poincar\'e inequality we have
\begin{align*}
	\|u\|=\|\nabla u \|_{p(\cdot)}\quad\text{ for all }u\in\Wpzero{p(\cdot)}.
\end{align*}

Suppose that $X$ is a Banach space and let $\ph \in C^1(X)$. We denote the critical set of $\ph$ by
\begin{align*}
	K_\ph=\left\{ u\in X: \ph'(u)=0\right\}.
\end{align*}
Moreover, we say that $\ph$ satisfies the ``Cerami condition'', C-condition for short, if every sequence $\{u_n\}_{n\in\N}\subseteq X$ such that $\{\ph(u_n)\}_{n\in\N}\subseteq \R$ is bounded and 
\begin{align*}
	\left(1+\|u_n\|_X\right)\ph'(u_n) \to 0\quad\text{in }X^* \text{ as }n\to \infty,
\end{align*}    
admits a strongly convergent subsequence. This is a compactness-type condition on the functional $\ph$ which compensates for the fact that the ambient space $X$ need not be locally compact being in general infinite dimensional. Applying this condition, one can prove a deformation theorem from which the minimax theorems for the critical values of $\ph$ follow. We refer to Papageorgiou-R\u{a}dulescu-Repov\v{s} \cite[Chapter 5]{13-Papageorgiou-Radulescu-Repovs-2019} and Struwe \cite[Chapter II]{Struwe-2008}.

Given $s\in (1,+\infty)$ we denote by $s'\in(1,+\infty)$ the conjugate exponent defined by
\begin{align*}
	\frac{1}{s}+\frac{1}{s'}=1.
\end{align*}
Furthermore, if $f\colon\Omega\times\R\to\R$ is a measurable function, then we denote by $N_f$ the Nemytskii (also called superposition) operator corresponding to $f$, that is,
\begin{align*}
	N_f(u)(\cdot)=f(\cdot,u(\cdot)) \quad\text{for all }u\in M(\Omega).
\end{align*}
Note that $x\to f(x,u(x))$ is measurable. We know that if $f\colon\Omega\times\R\to\R$ is a Carath\'eodory function, then $f(\cdot,\cdot)$ is jointly measurable, see Papageorgiou-Winkert \cite[p.\,106]{16-Papageorgiou-Winkert-2018}.

Now we are in the position to introduce our hypotheses on the data of problem \eqref{problem}.
\begin{enumerate}
	\item[H$_0$:]
		$p,q \in E_1\cap C^{0,1}(\close)$, $\eta\in C(\close)$, $q(x)<p(x)$, $0<\eta(x)<1$ for all $x\in \close$, $p_-<N$.
\end{enumerate}
\begin{enumerate}
	\item[H$_1$:]
		$f\colon\Omega\times\R\to\R$ is a Carath\'eodory function such that $f(x,0)=0$ for a.\,a.\,$x\in\Omega$ and
		\begin{enumerate}
			\item[(i)]
				there exists $a \in \Linf$ such that
				\begin{align*}
					0\leq f(x,s) \leq a(x) \l [ 1+s^{r-1}\r]
				\end{align*}
				for a.\,a.\,$x\in \Omega$, for all $s\geq 0$ and with $p_+<r<p_-^*$, where
				'\begin{align*}
					p_-^*=\frac{Np_-}{N-p_-};
				\end{align*}
			\item[(ii)]
				if $F(x,s)=\displaystyle\int_0^s f(x,t)\,dt$, then
				\begin{align*}
					\lim_{s\to +\infty} \frac{F(x,s)}{s^{p_+}}=+\infty \quad\text{uniformly for a.\,a.\,}x\in\Omega;
				\end{align*}
			\item[(iii)]
				there exists a function $\tau\in C(\close)$ such that
				\begin{align*}
					\tau(x) \in \l (\l(r-p_-\r)\frac{N}{p_-}, p^*(x) \r)\quad\text{ for all }x\in\close
				\end{align*}
				and
				\begin{align*}
					0<\gamma_0 \leq \liminf_{s\to +\infty} \frac{f(x,s)s-p_+F(x,s)}{s^{\tau(x)}} \quad \text{uniformly for a.\,a.\,}x\in\Omega;
				\end{align*}
			\item[(iv)]
				for every $\rho>0$ there exists $\hat{\xi}_\rho>0$ such that the function
				\begin{align*}
					s\to f(x,s)+\hat{\xi}_\rho s^{p(x)-1}
				\end{align*}
				is nondecreasing on $[0,\rho]$ for a.\,a.\,$x\in \Omega$.
		\end{enumerate}
\end{enumerate}

\begin{remark}
	Since we are interested in positive solutions and all the hypotheses above concern the positive semiaxis $\R_+=[0,+\infty)$, we may assume without any loss of generality that $f(x,s)= 0$ for a.\,a.\,$x\in\Omega$ and for all $s \leq 0$. Hypotheses H$_1$(ii), (iii) imply that $f(x,\cdot)$ is $(p_+-1)$-superlinear for a.\,a.\,$x\in\Omega$. However, this superlinearity condition on $f(x,\cdot)$ is not formulated by using the Ambrosetti-Rabinowitz condition which is common in the literature when dealing with superlinear problems, see Byun-Ko \cite{2-Byun-Ko-2017}, Saoudi-Ghanmi \cite{20-Saoudi-Ghanmi-2017} and Saoudi-Kratou-Alsadhan \cite{21-Saoudi-Kratou-Alsadhan-2016}. Here, instead of the Ambrosetti-Rabinowitz condition, we employ hypothesis H$_1$(iii) which is less restrictive and incorporates in our framework nonlinearities with ``slower'' growth near $+\infty$. For example, consider the functions
	\begin{align*}
		f_1(x,s)=(s+1)^{p_+-1}\ln (s+1) +s^{r_1(x)-1} \quad \text{for all }s\geq 0
	\end{align*}
	with $r_1\in E_1$, $r_1(x) \leq p(x)$ for all $x\in\close$ and
	\begin{align*}
		f_2(x,s)=
		\begin{cases}
			s^{\mu(x)-1}&\text{if }0 \leq s \leq 1,\\
			s^{p_+-1}\ln(s)+s^{r_2(x)-1}&\text{if } 1<s
		\end{cases}
	\end{align*}
	with $\mu, r_2\in E_1$ and $r_2(x) \leq p(x)$ for all $x\in\close$. These functions satisfy hypotheses H$_1$, but fail to satisfy the Ambrosetti-Rabinowitz condition, see, for example, Gasi\'nski-Papageorgiou \cite{7-Papageorgiou-Gasinski-2011}.
\end{remark}

The difficulty that we encounter when we study a singular problem is that the energy (Euler) functional of the problem is not $C^1$ because of the presence of the singular term. Hence, we cannot use the results of critical point theory. We need to find a way to bypass the singularity and deal with $C^1$-functionals. In the next section, we examine a purely singular problem and the solution of this problem will help us in bypassing the singularity.

\section{An auxiliary purely singular problem}

In this section we deal with the following purely singular anisotropic $(p,q)$-equation
\begin{align}\label{1}
	-\Delta_{p(\cdot)} u - \Delta_{q(\cdot)} u = u^{-\eta(x)} \quad \text{in }\Omega, \quad u\big|_{\partial\Omega}=0, \quad u>0.
\end{align}

\begin{proposition}\label{proposition_4}
	If hypotheses H$_0$ hold, then problem \eqref{1} admits a unique position solution $\overline{u}\in \interior$.
\end{proposition}

\begin{proof}
	Let $g \in \Lp{p(\cdot)}$ and let $0 < \eps \leq 1$. We consider the following Dirichlet problem
	\begin{align*}
	-\Delta_{p(\cdot)} u - \Delta_{q(\cdot)} u = \l[|g(x)|+\eps\r]^{-\eta(x)} \quad \text{in }\Omega, \quad u\big|_{\partial\Omega}=0, \quad u>0.
	\end{align*}
	Let $V\colon \Wpzero{p(\cdot)}\to \Wpzero{p(\cdot)}^*=W^{-1,p'(\cdot)}(\Omega)$ be the operator defined by
	\begin{align*}
		V(u)=A_{p(\cdot)}(u)+A_{q(\cdot)}(u) \quad\text{ for all } u\in\Wpzero{p(\cdot)}.
	\end{align*}
	This map is continuous and strictly monotone, see Proposition \ref{proposition_2}, hence maximal monotone as well. It is also coercive, see Proposition \ref{proposition_1}. Therefore, it is surjective, see Papageorgiou-R\u{a}dulescu-Repov\v{s} \cite[p.\,135]{13-Papageorgiou-Radulescu-Repovs-2019}. Since $[|g(\cdot)|+\eps]^{-\eta(\cdot)}\in \Linf$, there exists $u_\eps \in \Wpzero{p(\cdot)}, u_\eps\geq 0, u_\eps\neq 0$ such that
	\begin{align*}
		V(u_\eps)=\l[|g|+\eps\r]^{-\eta(\cdot)}.
	\end{align*}
	
	The strict monotonicity of $V$ implies the uniqueness of $u_\eps$. Thus, we can define the map $\beta\colon \Lp{p(\cdot)}\to \Lp{p(\cdot)}$ by setting
	\begin{align*}
		\beta(g)=u_\eps.
	\end{align*}
	Recall that $\Wpzero{p(\cdot)} \hookrightarrow \Lp{p(\cdot)}$ is compactly embedded. We claim that the map $\beta$ is continuous. So, let $g_n\to g$ in $\Lp{p(\cdot)}$ and let $u_\eps^n=\beta(g_n)$ with $n\in\N$. We have
	\begin{align}\label{3}
		\l\lan A_{p(\cdot)}\l(u_\eps^n\r),h\r\ran + \l\lan A_{q(\cdot)}\l(u_\eps^n\r),h\r\ran =\into \frac{h}{\l[|g_n|+\eps\r]^{\eta(x)}}\,dx
	\end{align}
	for all $h\in \Wpzero{p(\cdot)}$ and for all $n\in\N$.
	
	We choose $h=u_\eps^{n}\in \Wpzero{p(\cdot)}$ in \eqref{3} and obtain
	\begin{align*}
		\varrho_{p(\cdot)}\l(\nabla u_\eps^n\r)+\varrho_{p(\cdot)}\l(\nabla u_\eps^n\r) \leq \into \frac{u_\eps^n}{\eps^{\eta_+}}\,dx,
	\end{align*}
	which by Proposition \ref{proposition_1} implies that 
	\begin{align*}
		\l\{u_\eps^n\r\}_{n\in\N} \subseteq \Wpzero{p(\cdot)} \text{ is bounded}.
	\end{align*}
	So, we may assume that
	\begin{align}\label{4}
		u_\eps^n \weak \tilde{u}_\eps \quad\text{in }\Wpzero{p(\cdot)} 
		\quad\text{and}\quad
		u_\eps^n \to \tilde{u}_\eps \quad\text{in }\Lp{p(\cdot)}.
	\end{align}
	Now we choose $h= u_\eps^n-\tilde{u}_\eps \in \Wpzero{p(\cdot)}$ in \eqref{3}, pass to the limit as $n\to \infty$ and apply \eqref{4} which results in
	\begin{align*}
		\lim_{n\to\infty} \l [\l\lan A_{p(\cdot)}\l(u_\eps^n\r),u_\eps^n-\tilde{u}_\eps \r\ran+\l\lan A_{q(\cdot)}\l(u_\eps^n\r),u_\eps^n-\tilde{u}_\eps \r\ran\r]=0.
	\end{align*}
	Since $A_{q(\cdot)}(\cdot)$ is monotone, we have
	\begin{align*}
		\limsup_{n\to\infty} \l [\l\lan A_{p(\cdot)}\l(u_\eps^n\r),u_\eps^n-\tilde{u}_\eps \r\ran+\l\lan A_{q(\cdot)}\l(\tilde{u}_\eps\r),u_\eps^n-\tilde{u}_\eps \r\ran\r]\leq 0.
	\end{align*}
	Applying \eqref{4} gives
	\begin{align*}
		\limsup_{n\to\infty} \l\lan A_{p(\cdot)}\l(u_\eps^n\r),u_\eps^n-\tilde{u}_\eps \r\ran \leq 0
	\end{align*}
	and so, by Proposition \ref{proposition_2},
	\begin{align}\label{5}
		u_\eps^n \to \tilde{u}_\eps \quad \text{in }\Wpzero{p(\cdot)}.
	\end{align}
	Passing  to the limit in \eqref{3} as $n\to \infty$ and using \eqref{5} yields
	\begin{align*}
		\l\lan A_{p(\cdot)}\l(\tilde{u}_\eps\r),h\r\ran + \l\lan A_{q(\cdot)}\l(\tilde{u}_\eps\r),h\r\ran =\into \frac{h}{\l[|g|+\eps\r]^{\eta(x)}}\,dx
	\end{align*}
	for all $h\in \Wpzero{p(\cdot)}$. Hence, $\tilde{u}_\eps =\beta(g)$. 
	
	So, for the original sequence, we have 
	\begin{align*}
		u_\eps^n=\beta(g_n) \to \beta (g) =\tilde{u}_\eps,
	\end{align*}
	which shows that $\beta$ is continuous.
	
	From the argument above and recalling that $\Wpzero{p(\cdot)} \hookrightarrow \Lp{p(\cdot)}$ compactly, we see that $\overline{\beta(\Lp{p(\cdot)}} \subseteq \Lp{p(\cdot)}$ is compact. So, by the Schauder-Tychonov fixed point theorem, see Papageorgiou-R\u{a}dulescu-Repov\v{s} \cite[p.\,298]{13-Papageorgiou-Radulescu-Repovs-2019} we can find $\hat{u}_\eps\in\Wpzero{p(\cdot)}$ such that $\beta(\hat{u}_\eps)=\hat{u}_\eps$.
		
	From Fan-Zhao \cite{5-Fan-Zhao-1999}, see also Gasi\'nski-Papageorgiou \cite{7-Papageorgiou-Gasinski-2011} and Marino-Winkert \cite{10-Marino-Winkert-2020}, we have that $\hat{u}_\eps \in\Linf$. Then, from Tan-Fang \cite[Corollary 3.1]{22-Tan-Fang-2013}, we have $\hat{u}_\eps \in C^1_0(\close)\setminus \{0\}$. Finally, the anisotropic maximum principle of Zhang \cite{24-Zhang-2005}, see also Papageorgiou-Vetro-Vetro \cite{14-Papageorgiou-Vetro-Vetro-2020}, implies that $\hat{u}_\eps\in\interior$.
	
	{\bf Claim:} If $0<\eps'\leq \eps$, then $\hat{u}_\eps \leq \hat{u}_{\eps'}$.
	We have
	\begin{align}\label{6}
		-\Delta_{p(\cdot)} \hat{u}_{\eps'} -\Delta_{q(\cdot)} \hat{u}_{\eps'}= \frac{1}{\l[\hat{u}_{\eps'}+\eps'\r]^{\eta(x)}}
		\geq \frac{1}{\l[\hat{u}_{\eps'}+\eps\r]^{\eta(x)}} \quad\text{in }\Omega.
	\end{align}
	We introduce the Carath\'eodory function $k_\eps\colon \Omega\times\R\to\R$ defined by
	\begin{align}\label{7}
		k_\eps(x,s)=
		\begin{cases}
			\displaystyle\frac{1}{\l[s^++\eps\r]^{\eta(x)}} &\text{if } s \leq \hat{u}_{\eps'}(x),\\[4ex]
			\displaystyle\frac{1}{\l[\hat{u}_{\eps'}(x)+\eps\r]^{\eta(x)}} & \text{if } \hat{u}_{\eps'}(x) < s.
		\end{cases}
	\end{align}
	We set $K_\eps(x,s)=\int^s_0 k_\eps(x,t)\,dt$ and consider the $C^1$-functional $J_\eps\colon \Wpzero{p(\cdot)}\to \R$ defined by
	\begin{align*}
		J_\eps(u)=\into \frac{1}{p(x)} |\nabla u|^{p(x)}\,dx +\into \frac{1}{q(x)} |\nabla u|^{q(x)}\,dx-\into K_\eps(x,u)\,dx
	\end{align*}
	for all $u \in \Wpzero{p(\cdot)}$. From \eqref{7} it is clear that $J_\eps\colon \Wpzero{p(\cdot)}\to \R$ is coercive and by the compact embedding $\Wpzero{p(\cdot)}\hookrightarrow\Lp{r}$ we know that it is also sequentially weakly lower semicontinuous. Therefore, by the Weierstra\ss-Tonelli theorem, there exists $\hat{u}_\eps^*\in\Wpzero{p(\cdot)}$ such that 
	\begin{align}\label{8}
		J_\eps\l(\hat{u}_\eps^*\r)=\min \l[J_\eps(u)\,:\,u \in \Wpzero{p(\cdot)}\r].
	\end{align}
	
	Let $u\in \interior$ and choose $t \in (0,1)$ small enough so that $tu \leq\hat{u}_{\eps'}$, recall that $ \hat{u}_{\eps'}\in\interior$ and use Proposition 4.1.22 of Papageorgiou-R\u{a}dulescu-Repov\v{s} \cite{13-Papageorgiou-Radulescu-Repovs-2019}. Then, by \eqref{7}, we obtain
	\begin{align*}
		J_\eps(tu)
		&\leq \frac{t^{q_-}}{q_-} \l[\varrho_{p(\cdot)}(\nabla u)+\varrho_{q(\cdot)}(\nabla u)\r]
		-\into \frac{1}{1-\eta(x)}(tu)^{1-\eta(x)}\,dx\\
		& \leq c_1t^{q_-}-c_2t^{1-\eta_-}
	\end{align*}
	for some $c_1=c_1(u)>0$, $c_2=c_2(u)>0$ and $t\in (0,1)$. Choosing $t \in (0,1)$ even smaller if necessary, we see that
	\begin{align*}
		J_\eps(tu)<0,
	\end{align*}
	since $1-\eta_-<1<q_-$. Then, by \eqref{8}, because $\hat{u}_\eps^*\in\Wpzero{p(\cdot)}$ is the global minimizer of $J_\eps$, we conclude that
	\begin{align*}
		J_\eps\l(\hat{u}_\eps^*\r)<0=J_\eps(0)
	\end{align*}
	and so $\hat{u}_\eps^*\neq 0$. 
	
	From \eqref{8} we have $J'_\eps\l(\hat{u}_\eps^*\r)=0$ which means 
	\begin{align}\label{9}
		\l\lan A_{p(\cdot)}\l(\hat{u}_\eps^*\r),h\r\ran+\l\lan A_{q(\cdot)}\l(\hat{u}_\eps^*\r),h\r\ran=\into k_\eps \l(x,\hat{u}_\eps^*\r)\,dx
	\end{align}
	for all $h\in \Wpzero{p(\cdot)}$. Testing \eqref{9} with $h=-\l(\hat{u}_\eps^*\r)^-\in\Wpzero{p(\cdot)}$ we obtain
	\begin{align*}
		\varrho_{p(\cdot)}\l(\nabla \l(\hat{u}_\eps^*\r)^-\r) \leq 0,
	\end{align*}
	because of \eqref{7} which by Proposition \ref{proposition_1} implies that
	\begin{align*}
		\hat{u}_\eps^* \geq 0 \quad\text{and}\quad \hat{u}_\eps^*\neq 0.
	\end{align*}
	
	Now we choose $h= \l(\hat{u}_\eps^*-\hat{u}_{\eps'}\r)^+ \in\Wpzero{p(\cdot)}$ in \eqref{9}. Applying \eqref{7} and \eqref{6} gives
	\begin{align*}
		&\l\lan A_{p(\cdot)}\l(\hat{u}_\eps^*\r),\l(\hat{u}_\eps^*-\hat{u}_{\eps'}\r)^+\r\ran+\l\lan A_{q(\cdot)}\l(\hat{u}_\eps^*\r),\l(\hat{u}_\eps^*-\hat{u}_{\eps'}\r)^+\r\ran\\
		&=\into \frac{1}{\l[\hat{u}_{\eps'}+\eps\r]^{\eta(x)}} \l(\hat{u}_\eps^*-\hat{u}_{\eps'}\r)^+\,dx\\
		& \leq \l\lan A_{p(\cdot)}\l(\hat{u}_{\eps'}\r),\l(\hat{u}_\eps^*-\hat{u}_{\eps'}\r)^+\r\ran+\l\lan A_{q(\cdot)}\l(\hat{u}_{\eps'}\r),\l(\hat{u}_\eps^*-\hat{u}_{\eps'}\r)^+\r\ran.
	\end{align*}
	Hence, $\hat{u}_\eps^* \leq \hat{u}_{\eps'}$. So we have proved that
	\begin{align}\label{10}
		\hat{u}_\eps^* \in \l[0,\hat{u}_{\eps'}\r], \quad \hat{u}_\eps^* \neq 0.
	\end{align}
	
	From \eqref{10}, \eqref{7}, \eqref{9} and the first part of the proof we infer that $\hat{u}_\eps^*=\hat{u}_{\eps'}$ and so, by \eqref{10}, $\hat{u}_\eps \leq \hat{u}_{\eps'}$. This proves the Claim.
	
	Next we will let $\eps \to 0^+$ to produce a solution of the purely singular problem \eqref{1}. To this end, let $\eps_n\to 0^+$ and set $\hat{u}_n=\hat{u}_{\eps_n}$ for all $n \in \N$. We have
	\begin{align}\label{11}
		\l\lan A_{p(\cdot)}\l(\hat{u}_n\r),h\r\ran+\l\lan A_{q(\cdot)}\l(\hat{u}_n\r),h\r\ran
		=\into \frac{h}{\l[\hat{u}_n+\eps_n\r]^{\eta(x)}}\,dx
	\end{align}
	for all $h \in \Wpzero{p(\cdot)}$ and for all $n \in \N$. Choosing $h= \hat{u}_n\in\Wpzero{p(\cdot)}$ leads to
	\begin{align*}
		\varrho_{p(\cdot)}\l(\nabla \hat{u}_n\r) \leq \into \hat{u}_n^{1-\eta(x)}\,dx \quad\text{for all }n \in \N.
	\end{align*}
	Therefore, $\{\hat{u}_n\}_{n\in\N}\subseteq \Wpzero{p(\cdot)}$ is bounded.
	
	By passing to an appropriate subsequence if necessary, we may assume that
	\begin{align}\label{12}
		\hat{u}_n\weak \overline{u} \quad\text{in }\Wpzero{p(\cdot)}
		\quad\text{and}\quad
		\hat{u}_n\to \overline{u} \quad\text{in }\Lp{p(\cdot)}.
	\end{align}
	Now we choose $h= \hat{u}_n-\overline{u}\in\Wpzero{p(\cdot)}$. This yields
	\begin{align*}
		\begin{split}
			&\l\lan A_{p(\cdot)}\l(\hat{u}_n\r),\hat{u}_n-\overline{u}\r\ran+\l\lan A_{q(\cdot)}\l(\hat{u}_n\r),\hat{u}_n-\overline{u}\r\ran\\
			& = \into \frac{\hat{u}_n-\overline{u}}{\l[\hat{u}_n+\eps_n\r]^{\eta(x)}}\,dx
			 \leq  \into \frac{\hat{u}_n-\overline{u}}{\hat{u}_1^{\eta(x)}}\,dx\quad\text{for all }n\in \N,
		\end{split}
	\end{align*}
	due to the Claim. 
	
	Let $\hat{d}(x)=\dist(x,\partial\Omega)$ for all $x\in\close$. Using Lemma 14.16 of Gilbarg-Trudinger \cite[p.\,355]{8-Gilbarg-Trudinger-1998} we have that $\hat{d}\in\interior$. We can find $c_3>0$ such that $c_3\hat{d} \leq \hat{u}_1$, see Papageorgiou-R\u{a}dulescu-Repov\v{s} \cite[p.\,274]{12-Papageorgiou-Radulescu-Repovs-2020}. Then we have for all $h \in \Wpzero{p(\cdot)}$ that
	\begin{align*}
		\l|\into \frac{h}{\hat{u}_1^{\eta(x)}}\,dx\r| \leq c_4 \into \frac{|h|}{\hat{d}}\,dx \leq c_5 \|\nabla h\|_{p(\cdot)}
	\end{align*}
	for some $c_4,c_5>0$. Here we used the anisotropic Hardy inequality of Harjulehto-H\"{a}st\"{o}-Koskenoja \cite{Harjulehto-Hasto-Koskenoja-2005}. From Marino-Winkert \cite{10-Marino-Winkert-2020} (see also Ragusa-Tachikawa \cite{19-Ragusa-Tachikawa-2020}) we have that $\{\hat{u}_n\}_{n\in\N}\subseteq \Linf$ is bounded. Moreover by the lemma and its proof of Lazer-McKenna \cite{9-Lazer-McKenna-1991} we know that $\hat{u}_1^{-\eta(\cdot)}\in \Lp{1}$. So, from \eqref{12} and the dominated convergence theorem, it follows that
	\begin{align*}
		\into \frac{\hat{u}_n-\overline{u}}{\hat{u}_1^{\eta(x)}}\,dx \longrightarrow 0 \quad \text{as }n\to\infty.
	\end{align*}
	This implies
	\begin{align*}
		\limsup_{n\to\infty} \l[\l\lan A_{p(\cdot)}\l(\hat{u}_n\r),\hat{u}_n-\overline{u}\r\ran+\l\lan A_{q(\cdot)}\l(\hat{u}_n\r),\hat{u}_n-\overline{u}\r\ran\r] \leq 0,
	\end{align*}
	which by the monotonicity of $A_{q(\cdot)}$ and the $\Ss_+$-property of $A_{p(\cdot)}$ (see Proposition \ref{proposition_2} and the first part of the proof) leads to
	\begin{align}\label{16}
		\hat{u}_n \to \overline{u} \quad\text{in }\Wpzero{p(\cdot)} \quad\text{and}\quad \hat{u}_1\leq \overline{u}.
	\end{align}
	
	So, if we pass to the limit in \eqref{11} as $n\to \infty$ and use the Lebesgue dominated convergence theorem, we then obtain
	\begin{align*}
		\l\lan A_{p(\cdot)}\l(\overline{u}\r),h\r\ran
		+\l\lan A_{q(\cdot)}\l(\overline{u}\r),h\r\ran
		= \into \frac{h}{\overline{u}^{\eta(x)}}\,dx\quad\text{for all }h\in\Wpzero{p(\cdot)}.
	\end{align*}
	Since $\hat{u}_1\leq \overline{u}$, we see that $\overline{u}\in\Wpzero{p(\cdot)}$ is a positive solution of \eqref{1}. From Marino-Winkert \cite{10-Marino-Winkert-2020} we know that $\overline{u} \in \Linf$ and so we conclude that $\overline{u} \in \interior$, see Zhang \cite{24-Zhang-2005} and \eqref{16}.
	
	Finally, note that the function $\overset{\circ}{\R}_+\ni s \to s^{-\eta(x)}$ is strictly decreasing. Therefore, the positive solution $\overline{u}\in \interior$ is unique.
\end{proof}

In the next section we will use this solution to bypass the singularity and deal with $C^1$-functionals on which we can apply the results of critical point theory.

\section{Positive solutions}

We introduce the following two sets
\begin{align*}
	\mathcal{L}&=\left\{\lambda>0: \text{problem \eqref{problem} has a positive solution}\right\},\\
	\mathcal{S}_\lambda&=\left\{u: u\text{ is a positive solution of problem \eqref{problem}}\right\}.
\end{align*}

\begin{proposition}\label{proposition_5}
	If hypotheses H$_0$ and H$_1$ hold, then $\mathcal{L}\neq \emptyset$.
\end{proposition}

\begin{proof}
	Let $\overline{u} \in \interior$ be the unique positive solution of problem \eqref{1}, see Proposition \ref{proposition_4}. By the anisotropic Hardy inequality, see Harjulehto-H\"{a}st\"{o}-Koskenoja \cite{Harjulehto-Hasto-Koskenoja-2005}, we know that $\overline{u}^{-\eta(\cdot)}h\in\Lp{1}$ for all $h\in \Wpzero{p(\cdot)}$. Hence, $\overline{u}^{-\eta(\cdot)}\in W^{1,p'(\cdot)}(\Omega)=\Wpzero{p(\cdot)}^*$.

	We consider the following auxiliary Dirichlet problem
	\begin{align}\label{17}
		-\Delta_{p(\cdot)} u - \Delta_{q(\cdot)} u = \overline{u}^{-\eta(x)}+1 \quad \text{in }\Omega, \quad u\big|_{\partial\Omega}=0, \quad u>0.
	\end{align}
	As in the proof of Proposition \ref{proposition_4}, exploiting the surjectivity and the strict monotonicity of the operator $V$, we infer that problem \eqref{17} admits a unique positive solution $\tilde{u} \in\Wpzero{p(\cdot)}$.
	
	Since $\overline{u}^{-\eta(\cdot)}\leq c_6 \hat{d}^{-\eta(\cdot)}$ for some $c_6>0$, from Theorem B.1 of Saoudi-Ghanmi \cite{20-Saoudi-Ghanmi-2017} we have 
	\begin{align*}
		\tilde{u} \in \interior.
	\end{align*}
	
	From the weak comparison principle, see Tolksdorf \cite{23-Tolksdorf-1983}, we have that
	\begin{align}\label{18}
		\overline{u} \leq \tilde{u}.
	\end{align}
	
	Let $\lambda_0 =\frac{1}{\|N_f(\tilde{u})\|_\infty}$, see hypothesis H$_1$(i). For $\lambda \in (0,\lambda_0]$ we have that
	\begin{align}\label{19}
		\lambda f\l(x,\tilde{u}(x)\r) \leq 1 \quad\text{for a.\,a.\,}x\in\Omega.
	\end{align}
	Applying \eqref{18} and \eqref{19} we get
	\begin{align}\label{20}
		-\Delta_{p(\cdot)}\tilde{u}-\Delta_{q(\cdot)} \tilde{u}
		=\overline{u}^{-\eta(x)}+1
		\geq \tilde{u}^{-\eta(x)}+\lambda f\l(x,\tilde{u}(x)\r)\quad \text{in }\Omega.
	\end{align}
	
	We introduce the Carath\'eodory function $i_\lambda\colon\Omega\times\overset{\circ}{\R}_+\to\overset{\circ}{\R}_+$ defined by
	\begin{align}\label{21}
		i_\lambda(x,s)=
		\begin{cases}
			\overline{u}(x)^{-\eta(x)}+\lambda f\l(x,\overline{u}(x)\r)&\text{if } s<\overline{u}(x),\\
			s^{-\eta(x)}+\lambda f\l(x,s\r)&\text{if } \overline{u}(x)\leq s\leq \tilde{u}(x),\\
			\tilde{u}(x)^{-\eta(x)}+\lambda f\l(x,\tilde{u}(x)\r)&\text{if } \tilde{u}(x)<s.
		\end{cases}
	\end{align}
	We set $I_\lambda(x,s)=\int^s_0i_\lambda(x,t)\,dt$ and consider the $C^1$-functional $\psi_\lambda\colon \Wpzero{p(\cdot)}\to\R$ defined by
	\begin{align*}
		\psi_\lambda(u)=\into \frac{1}{p(x)}|\nabla u|^{p(x)}\,dx+\into \frac{1}{q(x)}|\nabla u|^{q(x)}\,dx-\into I_\lambda(x,u)\,dx
	\end{align*}
	for all $u \in \Wpzero{p(\cdot)}$. Evidently, $\psi_\lambda$ is coercive due to \eqref{21} and it is sequentially weakly lower semicontinuous. So, we can find $u_\lambda \in \Wpzero{p(\cdot)}$ such that
	\begin{align*}
		\psi_\lambda(u_\lambda)=\min\l[\psi_\lambda(u)\,:\, u \in \Wpzero{p(\cdot)}\r].
	\end{align*}
	From this we know that $\psi'_\lambda(u_\lambda)=0$ and so,
	\begin{align}\label{22}
		\l \lan A_{p(\cdot)}\l(u_\lambda\r),h\r\ran+\l \lan A_{q(\cdot)}\l(u_\lambda\r),h\r\ran=\into i_\lambda \l(x,u_\lambda\r)h\,dx
	\end{align}
	for all $h \in \Wpzero{p(\cdot)}$. First we choose $h=\l(\overline{u}-u_\lambda\r)^+ \in \Wpzero{p(\cdot)}$ in \eqref{22}. Then, by \eqref{21}, $f \geq 0$ and Proposition \ref{proposition_4} it follows that
	\begin{align*}
		&\l \lan A_{p(\cdot)}\l(u_\lambda\r),\l(\overline{u}-u_\lambda\r)^+\r\ran+\l \lan A_{q(\cdot)}\l(u_\lambda\r),\l(\overline{u}-u_\lambda\r)^+\r\ran\\
		&=\into i_\lambda \l(x,u_\lambda\r)\l(\overline{u}-u_\lambda\r)^+\,dx\\
		&=\into \l[\overline{u}^{-\eta(x)}+\lambda f\l(x,\overline{u}\r)\r] \l(\overline{u}-u_\lambda\r)^+\,dx\\
		&\geq \into \overline{u}^{-\eta(x)}\l(\overline{u}-u_\lambda\r)^+\,dx\\
		&=\l \lan A_{p(\cdot)}\l(\overline{u}\r),\l(\overline{u}-u_\lambda\r)^+\r\ran+\l \lan A_{q(\cdot)}\l(\overline{u}\r),\l(\overline{u}-u_\lambda\r)^+\r\ran.
	\end{align*}
	Therefore, $\overline{u} \leq u_\lambda$.
	
	Next, we test \eqref{22} with $h=\l(u_\lambda-\tilde{u}\r)^+ \in \Wpzero{p(\cdot)}$. As before, by \eqref{21} and \eqref{20}, we have
	\begin{align*}
		&\l \lan A_{p(\cdot)}\l(u_\lambda\r),\l(u_\lambda-\tilde{u}\r)^+\r\ran+\l \lan A_{q(\cdot)}\l(u_\lambda\r),\l(u_\lambda-\tilde{u}\r)^+\r\ran\\
		&=\into i_\lambda \l(x,u_\lambda\r)\l(u_\lambda-\tilde{u}\r)^+\,dx\\
		&=\into \l[\tilde{u}^{-\eta(x)}+\lambda f\l(x,\tilde{u}\r)\r] \l(u_\lambda-\tilde{u}\r)^+\,dx\\
		&\leq \l \lan A_{p(\cdot)}\l(\tilde{u}\r),\l(u_\lambda-\tilde{u}\r)^+\r\ran+\l \lan A_{q(\cdot)}\l(\tilde{u}\r),\l(u_\lambda-\tilde{u}\r)^+\r\ran.
	\end{align*}
	Hence, $u_\lambda \le \tilde{u}$. So, we have proved that
	\begin{align}\label{23}
		u_\lambda\in \l[\overline{u},\tilde{u}\r].
	\end{align}
	Then, from \eqref{23}, \eqref{21} and \eqref{22}, it follows that
	\begin{align*}
		u_\lambda \in \mathcal{S}_\lambda \quad\text{for all }\lambda \in \l(0,\lambda_0\r].
	\end{align*}
	Thus, $\l(0,\lambda_0\r] \subseteq\mathcal{L}\neq \emptyset$.
\end{proof}

We want to determine the regularity of the elements of the solution set $\mathcal{S}_\lambda$. To this end, we first establish a lower bound for the elements of $\mathcal{S}_\lambda$.

\begin{proposition}\label{proposition_6}
	If hypotheses H$_0$, H$_1$ hold and $\lambda \in \mathcal{L}$, then $\overline{u}\leq u$ for all $u \in \mathcal{S}_\lambda$.
\end{proposition}

\begin{proof}
	Let $u \in \mathcal{S}_\lambda$. We introduce the Carath\'eodory function $b\colon \Omega \times \overset{\circ}{\R}_+ \to \overset{\circ}{\R}_+$ defined by
	\begin{align}\label{24}
		b(x,s)=
		\begin{cases}
			s^{-\eta(x)}&\text{if }0<s<u(x),\\
			u(x)^{-\eta(x)}&\text{if } u(x)<s.
		\end{cases}
	\end{align}
	
	We consider the following Dirichlet problem
	\begin{align}\label{25}
		-\Delta_{p(\cdot)} u - \Delta_{q(\cdot)} u = b(x,u) \quad \text{in }\Omega, \quad u\big|_{\partial\Omega}=0, \quad u>0.
	\end{align}
	As in the proof of Proposition \ref{proposition_4}, using approximations and fixed point theory, we can show that problem \eqref{25} has a positive solution $\overline{u}_0 \in \Wpzero{p(\cdot)}$. Applying \eqref{24}, $f\geq 0$ and $u \in \mathcal{S}_\lambda$ yields
	\begin{align*}
		&\l \lan A_{p(\cdot)}\l(\overline{u}_0\r),\l(\overline{u}_0-u\r)^+\r\ran+\l \lan A_{q(\cdot)}\l(\overline{u}_0\r),\l(\overline{u}_0-u\r)^+\r\ran\\
		&=\into b\l(x,\overline{u}_0\r)\l(\overline{u}_0-u\r)^+\,dx\\
		&= \into u^{-\eta(x)} \l(\overline{u}_0-u\r)^+\,dx\\
		& \leq \into \l[u^{-\eta(x)}+\lambda f(x,u)\r] \l(\overline{u}_0-u\r)^+\,dx\\
		&=\l \lan A_{p(\cdot)}\l(u\r),\l(\overline{u}_0-u\r)^+\r\ran+\l \lan A_{q(\cdot)}\l(u\r),\l(\overline{u}_0-u\r)^+\r\ran.
	\end{align*}
	Therefore, we have
	\begin{align}\label{26}
		\overline{u}_0 \leq u.
	\end{align}
	Then, \eqref{26}, \eqref{24}, \eqref{25} and Proposition \ref{proposition_4} imply that 
	\begin{align*}
		\overline{u}_0=\overline{u} \in \interior.
	\end{align*}
	This shows that $\overline{u} \leq u$ for all $u \in \mathcal{S}_\lambda$, see \eqref{26}.
\end{proof}

Using this lower bound and the anisotropic regularity theory of Saoudi-Ghanmi \cite{20-Saoudi-Ghanmi-2017}, we can have the regularity properties of the elements of $\mathcal{S}_\lambda$.

\begin{proposition}\label{proposition_7}
	If hypotheses H$_0$, H$_1$ hold and $\lambda\in \mathcal{L}$, then $\emptyset\neq \mathcal{S}_\lambda \subseteq \interior$.
\end{proposition}

Next we prove a structural property of $\mathcal{L}$, namely, we show that $\mathcal{L}$ is connected, so an interval.

\begin{proposition}\label{proposition_8}
	If hypotheses H$_0$, H$_1$ hold, $\lambda\in\mathcal{L}$ and $\mu \in (0,\lambda)$, then $\mu \in \mathcal{L}$.
\end{proposition}

\begin{proof}
	Let $u_\lambda\in\mathcal{S}_\lambda\subseteq \interior$, see Proposition \ref{proposition_7}. We introduce the\\Carath\'eodory function $e_\mu\colon \Omega \times \overset{\circ}{\R}_+ \to\overset{\circ}{\R}_+$ defined by
	\begin{align}\label{28}
		e_\mu(x,s)=
		\begin{cases}
			\overline{u}(x)^{-\eta(x)}+\mu f\l(x,\overline{u}(x)\r)&\text{if } s<\overline{u}(x),\\
			s^{-\eta(x)}+\mu f\l(x,s\r)&\text{if } \overline{u}(x)\leq s \leq u_\lambda(x),\\
			u_\lambda(x)^{-\eta(x)}+\mu f\l(x,u_\lambda(x)\r)&\text{if } u_\lambda(x) <s.
		\end{cases}
	\end{align}
	We set $E_\mu(x,s)=\int^s_0e_\mu(x,t)\,dt$ and consider the $C^1$-functional $\sigma_\mu\colon \Wpzero{p(\cdot)}\to\R$ defined by
	\begin{align*}
		\sigma_\mu(u)
		=\into \frac{1}{p(x)}|\nabla u|^{p(x)}\,dx+\into \frac{1}{q(x)}|\nabla u|^{q(x)}\,dx-\into E_\mu(x,u)\,dx
	\end{align*}
	for all $u \in \Wpzero{p(\cdot)}$. It is clear that $\sigma_\mu$ is coercive because of \eqref{28} and it is sequentially weakly lower semicontinuous. So, there exists $u_\mu \in \Wpzero{p(\cdot)}$ such that
	\begin{align*}
		\sigma_\mu(u_\mu)=\min\l[\sigma_\mu(u)\,:\, u \in \Wpzero{p(\cdot)}\r].
	\end{align*}
	That means $\sigma'_\mu(u_\mu)=0$ and so,
	\begin{align}\label{29}
		\l \lan A_{p(\cdot)}\l(u_\mu\r),h\r\ran+\l \lan A_{q(\cdot)}\l(u_\mu\r),h\r\ran=\into e_\mu \l(x,u_\mu\r)h\,dx
	\end{align}
	for all $h \in \Wpzero{p(\cdot)}$. If we choose $h=\l(\overline{u}-u_\mu\r)^+\in\Wpzero{p(\cdot)}$ in \eqref{29} we can show that $\overline{u} \leq u_\mu$, see the proof of Proposition \ref{proposition_5}. Next, we choose $h=\l(u_\mu-u_\lambda\r)^+\in\Wpzero{p(\cdot)}$ in \eqref{29}. Then, by \eqref{28}, $f \geq 0$, $\mu<\lambda$ and $u_\lambda\in\mathcal{S}_\lambda$, we obtain
	\begin{align*}
		&\l \lan A_{p(\cdot)}\l(u_\mu\r),\l(u_\mu-u_\lambda\r)^+\r\ran+\l \lan A_{q(\cdot)}\l(u_\mu\r),\l(u_\mu-u_\lambda\r)^+\r\ran\\
		&=\into e_\mu \l(x,u_\mu\r)\l(u_\mu-u_\lambda\r)^+\,dx\\
		&=\into \l[u_\lambda^{-\eta(x)}+\mu f\l(x,u_\lambda\r)\r]\l(u_\mu-u_\lambda\r)^+\,dx\\
		& \leq \into \l[u_\lambda^{-\eta(x)}+\lambda f\l(x,u_\lambda\r)\r]\l(u_\mu-u_\lambda\r)^+\,dx\\
		& =\l \lan A_{p(\cdot)}\l(u_\lambda \r),\l(u_\mu-u_\lambda\r)^+\r\ran+\l \lan A_{q(\cdot)}\l(u_\lambda\r),\l(u_\mu-u_\lambda\r)^+\r\ran.
	\end{align*}
	Hence, $u_\mu \leq u_\lambda$. Therefore we have 
	\begin{align}\label{30}
		u_\mu \in \l[\overline{u},u_\lambda\r].
	\end{align}
	From \eqref{30}, \eqref{28} and \eqref{29} it follows that
	\begin{align*}
		u_\mu \in \mathcal{S}_\mu \subseteq \interior \quad\text{and so}\quad \mu \in \mathcal{L}.
	\end{align*}
\end{proof}

From Proposition \ref{proposition_8} and its proof we have the following corollary.

\begin{corollary}\label{31}
	If hypotheses H$_0$, H$_1$ hold and if $\lambda \in\mathcal{L}, u_\lambda \in\mathcal{S}_\lambda \subseteq \interior$ and $0<\mu<\lambda$, then $\mu\in\mathcal{L}$ and there exists $u_\mu \in \mathcal{S}_\mu \subseteq \interior$ such that $u_\mu \leq u_\lambda$.
\end{corollary}

In the next proposition we are going to improve the assertion of Corollary \ref{31}.

\begin{proposition}\label{proposition_9}
	If hypotheses H$_0$, H$_1$ hold and if $\lambda \in\mathcal{L}, u_\lambda \in\mathcal{S}_\lambda \subseteq \interior$ and $0<\mu<\lambda$, then $\mu\in\mathcal{L}$ and there exists $u_\mu \in \mathcal{S}_\mu \subseteq \interior$ such that
	\begin{align*}
		u_\lambda-u_\mu \in \interior.
	\end{align*}
\end{proposition}

\begin{proof}
	From Corollary \ref{31} we already know that $\mu \in \mathcal{L}$ and that there exists $u_\mu \in \mathcal{S}_\mu \subseteq \interior$ such that
	\begin{align}\label{32}
		u_\mu \leq u_\lambda.
	\end{align}
	
	Let $\rho=\|u_\lambda\|_\infty$ and let $\hat{\xi}_\rho>0$ be as postulated by hypothesis H$_1$(iv). Applying $u_\mu \in \mathcal{S}_\mu$, \eqref{32}, hypothesis H$_1$(iv), $f \geq 0$, $\mu<\lambda$ and $u_\lambda \in \mathcal{S}_\lambda$ gives
	\begin{align}\label{33}
		\begin{split}
			&-\Delta_{p(\cdot)}u_\mu-\Delta_{q(\cdot)}u_\mu+\mu\hat{\xi}_\rho u_\mu^{p(x)-1}-u_\mu^{-\eta(x)}\\
			&=\mu \l[f \l(x,u_\mu\r)+\hat{\xi}_\rho u_\mu^{p(x)-1}\r]\\
			& \leq \mu \l[f \l(x,u_\lambda\r)+\hat{\xi}_\rho u_\lambda^{p(x)-1}\r]\\
			& \leq \lambda  f \l(x,u_\lambda\r)+\mu \hat{\xi}_\rho u_\lambda^{p(x)-1}\\
			&=-\Delta_{p(\cdot)}u_\lambda-\Delta_{q(\cdot)}u_\lambda+\mu\hat{\xi}_\rho u_\lambda^{p(x)-1}-u_\lambda^{-\eta(x)}.
		\end{split}
	\end{align}
	Note that since $u_\mu \in \interior$, $f \geq 0$ and $\mu<\lambda$, we have
	\begin{align*}
		(\lambda-\mu) \l[N_f(u_\mu)+\hat{\xi}_\rho u_\mu^{p(\cdot)-1}\r] \succeq 0.
	\end{align*}
	Hence, from \eqref{33} and Proposition \ref{proposition_3}(a), we infer that
	\begin{align*}
		u_\lambda-u_\mu \in \interior.
	\end{align*} 
\end{proof}

We set $\lambda^*=\sup \mathcal{L}$.

\begin{proposition}
	If hypotheses H$_0$, H$_1$ hold, then $\lambda^*<+\infty$.
\end{proposition}

\begin{proof}
	Hypotheses H$_1$(i), (ii) and (iii) imply that we can find $\hat{\lambda}>0$ such that
	\begin{align}\label{34}
		s^{-\eta(x)}+\hat{\lambda}f(x,s)\geq s^{p(x)-1} \quad\text{for a.\,a.\,}x\in \Omega \text{ and for all }s>0.
	\end{align}
	Let $\lambda>\hat{\lambda}$ and suppose that $\lambda \in \mathcal{L}$. We can find $u\in \mathcal{S}_\lambda\subseteq \interior$ and from Proposition \ref{proposition_6} we have $\overline{u} \leq u$. Let $\Omega_0\subseteq \Omega$ be an open subset with $C^2$-boundary, $\overline{\Omega}_0\subseteq \Omega$ and $m_0=\min_{x\in \overline{\Omega}_0}u(x) \leq 1$. Note that since $u \in \interior$ we have $0<m_0$. Let $\delta \in (0,1)$ be small and set $m_0^\delta=m_0+\delta$. Note that
	\begin{align}\label{35}
		0 \leq \frac{1}{m_0^{\eta(x)}}-\frac{1}{\l(m_0^\delta\r)^{\eta(x)}}
		\leq \frac{\delta^{\eta(x)}}{m_0^{2\eta(x)}} \leq \frac{\delta^{\eta_-}}{m_0^{2\eta_+}}\quad\text{for all }x\in \close.
	\end{align}
	Let $\rho=\|u\|_\infty$ and let $\hat{\xi}_\rho>0$ be as postulated by hypothesis H$_1$(iv). Then, by applying \eqref{35}, \eqref{34}, $m_0\leq 1$, $\delta>0$ small enough, $f \geq 0$ and $\lambda>\hat{\lambda}$ we obtain
	\begin{align*}
		& -\Delta_{p(\cdot)}m_0^\delta-\Delta_{q(\cdot)}m_0^\delta+\hat{\lambda} \hat{\xi}_\rho \l(m_0^\delta\r)^{p(x)-1}-\l(m_0^\delta\r)^{-\eta(x)}\\
		& \leq \hat{\lambda} \hat{\xi}_\rho m_0^{p(x)-1}+\chi(\delta)-m_0^{-\eta(x)} \quad\text{with }\chi(\delta)\to 0^+ \text{ as }\delta\to 0^+,\\
		&\leq \l[\hat{\lambda} \hat{\xi}_\rho+1\r] m_0^{p(x)-1}+\chi(\delta)-m_0^{-\eta(x)}\\
		& \leq \hat{\lambda} \l[f(x,m_0)+ \hat{\xi}_\rho m_0^{p(x)-1}\r]+\chi(\delta)-m_0^{-\eta_+}\\
		& < \hat{\lambda} \l[f(x,u)+ \hat{\xi}_\rho u^{p(x)-1}\r]\\
		& \leq \lambda f(x,u)+ \hat{\lambda}\hat{\xi}_\rho u^{p(x)-1}\\
		& =-\Delta_{p(\cdot)}u-\Delta_{q(\cdot)}u+\hat{\lambda} \hat{\xi}_\rho u^{p(x)-1}-u^{-\eta(x)} \qquad \text{in }\Omega_0.
	\end{align*}
	Then, by Proposition \ref{proposition_3}(b), we get $m_0^\delta<u(x)$ for all $x \in \Omega_0$ and for all $\delta \in (0,1)$ small enough. This contradicts the definition of $m_0$. Therefore, $\lambda^* \leq \hat{\lambda}<+\infty$. 
\end{proof}

Next we are going to prove that we have multiple solutions for all $\lambda \in (0,\lambda^*)$.

\begin{proposition}\label{proposition_11}
	If hypotheses H$_0$, H$_1$ hold and $\lambda \in (0,\lambda^*)$, then problem \eqref{problem} has at least two positive solutions
	\begin{align*}
		u_0, \hat{u} \in \interior \text{ with }u_0 \neq \hat{u}.
	\end{align*}
\end{proposition}

\begin{proof}
	Let $0<\lambda<\vartheta<\lambda^*$. On account of Proposition \ref{proposition_9} we can find $u_\vartheta \in \mathcal{S}_\vartheta\subseteq \interior$ and $u_0 \in \mathcal{S}_\lambda \subseteq \interior$ such that 
	\begin{align}\label{36}
		u_\vartheta-u_0 \in \interior.
	\end{align}
	Also from Proposition \ref{proposition_6} we have
	\begin{align}\label{37}
		\overline{u} \leq u_0.
	\end{align}
	Let $\rho=\|u_0\|_\infty$ and let $\hat{\xi}_\rho>0$ be as postulated by hypothesis H$_1$(iv). Then, using $f \geq 0$, \eqref{37}, hypothesis H$_1$(iv) and $u_0 \in \mathcal{S}_\lambda$, we obtain
	\begin{align}\label{38}
		\begin{split} 
			&-\Delta_{p(\cdot)}\overline{u}-\Delta_{q(\cdot)} \overline{u}+\lambda \hat{\xi}_\rho \overline{u}^{p(x)-1}-\overline{u}^{-\eta(x)}\\
			& \leq \lambda \l[f\l(x,\overline{u}\r)+\hat{\xi}_\rho \overline{u}^{p(x)-1}\r]\\
			& \leq \lambda \l[f\l(x,u_0\r)+\hat{\xi}_\rho u_0^{p(x)-1}\r]\\
			&=-\Delta_{p(\cdot)}u_0-\Delta_{q(\cdot)} u_0+\lambda \hat{\xi}_\rho u_0^{p(x)-1}-u_0^{-\eta(x)}\quad \text{in }\Omega.
		\end{split} 
	\end{align}
	Note that $0 \preceq \hat{\xi}_\rho u_0^{p(x)-1}$ since $u_0 \in \interior$. So, from \eqref{38} and Proposition \ref{proposition_3}(a), we get that
	\begin{align}\label{39}
		u_0-\overline{u} \in \interior.
	\end{align}
	From \eqref{36} and \eqref{39} it follows that
	\begin{align}\label{40}
		u_0 \in \sideset{}{_{C^1_0(\close)}}\ints \l[\overline{u}, u_\vartheta\r].
	\end{align}
	
	We introduce the Carath\'eodory function $j_\lambda\colon \Omega\times\overset{\circ}{\R}_+\to\overset{\circ}{\R}_+$ defined by
	\begin{align}\label{41}
		j_\lambda(x,s)=
		\begin{cases}
			\overline{u}(x)^{-\eta(x)}+\lambda f\l(x,\overline{u}(x)\r) &\text{if }s \leq \overline{u}(x),\\
			s^{-\eta(x)}+\lambda f\l(x,s\r) &\text{if }\overline{u}(x)<s.
		\end{cases}
	\end{align}
	Moreover, we introduce the truncation of $j_\lambda(x,\cdot)$ at $u_\vartheta(x)$, namely, the Carath\'eodory function $\hat{j}_\lambda\colon\Omega\times\overset{\circ}{\R}_+\to\overset{\circ}{\R}_+$ defined by
	\begin{align}\label{42}
		\hat{j}_\lambda(x,s)=
		\begin{cases} 
			j_\lambda(x,s) & \text{if } s \leq u_\vartheta(x),\\
			j_\lambda\l(x,u_\vartheta(x)\r) & \text{if } u_\vartheta(x)<s.
		\end{cases} 
	\end{align}
	We set $J_\lambda(x,s)=\int^s_0 j_\lambda(x,t)\,dt$ and $\hat{J}_\lambda(x,s)=\int^s_0 \hat{j}_\lambda(x,t)\,dt$ and consider the $C^1$-functionals $w_{\lambda}, \hat{w}_\lambda\colon \Wpzero{p(\cdot)}\to \R$ defined by
	\begin{align*}
		w_\lambda(u)
		&=\into \frac{1}{p(x)}|\nabla u|^{p(x)}\,dx +\into \frac{1}{q(x)}|\nabla u|^{q(x)}\,dx-\into J_\lambda (x,u)\,dx,\\[1ex]
		\hat{w}_\lambda(u)&=\into \frac{1}{p(x)}|\nabla u|^{p(x)}\,dx +\into \frac{1}{q(x)}|\nabla u|^{q(x)}\,dx-\into \hat{J}_\lambda (x,u)\,dx
	\end{align*}
	for all $u \in \Wpzero{p(\cdot)}$.
	
	From \eqref{41} and \eqref{42} it is clear that
	\begin{align}\label{43}
		w_\lambda \big|_{[0,u_\vartheta]}=\hat{w}_\lambda \big|_{[0,u_\vartheta]}
		\quad\text{and}\quad
		w'_\lambda \big|_{[0,u_\vartheta]}=\hat{w}'_\lambda \big|_{[0,u_\vartheta]}.
	\end{align}
	Moreover, applying \eqref{41} and \eqref{42}, we can easily show that
	\begin{align}\label{44}
		K_{w_\lambda} \subseteq [\overline{u})\cap \interior
		\quad\text{and}\quad
		K_{\hat{w}_\lambda}\subseteq [\overline{u},u_\vartheta]\cap \interior.
	\end{align}
	On account of \eqref{42} and \eqref{44}, we see that we may assume that
	\begin{align}\label{45}
		K_{\hat{w}_\lambda}=\{u_0\}.
	\end{align}
	Otherwise we already have a second positive smooth solution for problem \eqref{problem} and so we are done, see \eqref{42} and \eqref{44}.
	
	From \eqref{42} we see that the functional $\hat{w}_\lambda\colon \Wpzero{p(\cdot)}\to \R$ is coercive and it is easy to check that it is sequentially weakly lower semicontinuous. Hence, its global minimizer $\hat{u}_0\in\Wpzero{p(\cdot)}$ exists, that is,
	\begin{align*}
		\hat{w}_\lambda\l(\hat{u}_0\r)=\min \l[\hat{w}_\lambda(u)\,:\,u\in\Wpzero{p(\cdot)}\r].
	\end{align*}
	From \eqref{45} we conclude that $\hat{u}_0=u_0$. From \eqref{40} and \eqref{43} it follows that $u_0$ is a local $C^1_0(\close)$-minimizer of $w_\lambda$, Hence
	\begin{align}\label{46}
		u_0 \text{ is a local $\Wpzero{p(\cdot)}$-minimizer of $w_\lambda$},
	\end{align}
	see Tan-Fang \cite{22-Tan-Fang-2013} and Gasi\'nski-Papageorgiou \cite{7-Papageorgiou-Gasinski-2011}. From \eqref{41} and \eqref{44} we see that we can assume that
	\begin{align}\label{47}
		K_{w_\lambda} \text{ is finite}.
	\end{align}
	Otherwise we already have an infinity of positive smooth solutions for problem \eqref{problem} and so we are done.
	
	Then, from \eqref{46}, \eqref{47} and Theorem 5.7.4 of Papageorgiou-R\u{a}dulescu-Repov\v{s} \cite[p.\,449]{13-Papageorgiou-Radulescu-Repovs-2019} we know that there exists $\rho \in (0,1)$ small such that
	\begin{align}\label{48}
		w_\lambda(u_0)<\inf \l[w_\lambda(u)\,:\, \|u-u_0\|=\rho \r]=m_\lambda.
	\end{align}
	
	On account of hypothesis H$_1$(ii), if $u \in \interior$, then
	\begin{align}\label{49}
		w_\lambda(tu)\to -\infty \quad\text{as }t \to +\infty.
	\end{align}
	
	In order to apply the mountain pass theorem we only need to show that the functional $w_\lambda$ satisfies the C-condition.
	
	{\bf Claim:} $w_\lambda$ fulfills the C-condition.
	
	We consider the sequence $\{u_n\}_{n\in\N}\subseteq \Wpzero{p(\cdot)}$ such that
	\begin{align}
		|w_\lambda(u_n)|\leq c_7\quad \text{for some }c_7>0 \text{ and for all }n\in\N,\label{50}\\
		(1+\|u_n\|)w'_\lambda(u_n) \to 0 \text{ in } W^{-1,p'(\cdot)}(\Omega) \text{ as }n\to \infty.\label{51}
	\end{align}
	From \eqref{51} we have
	\begin{align}\label{52}
		\l|\l\lan A_{p(\cdot)}(u_n),h\r\ran+\l\lan A_{q(\cdot)}(u_n),h\r\ran - \into j_\lambda (x,u_n)h\,dx \r| \leq \frac{\eps_n \|h\|}{1+\|u_n\|}
	\end{align}
	for all $h\in\Wpzero{p(\cdot)}$ with $\eps_n\to 0^+$. Choosing $h=-u_n^-\in\Wpzero{p(\cdot)}$ in \eqref{52}, recalling that $\overline{u}^{-\eta(\cdot)}h\in \Lp{1}$ for all $h\in \Wpzero{p(\cdot)}$, see Harjulehto-H\"{a}st\"{o}-Koskenoja \cite{Harjulehto-Hasto-Koskenoja-2005}, and applying \eqref{41} leads to
	\begin{align*}
		\varrho_{p(\cdot)}(\nabla u_n^-)+\varrho_{q(\cdot)}(\nabla u_n^-)  \leq c_8 \l\|u_n^-\r\| \quad\text{for some }c_8>0 \text{ and for all }n\in\N,
	\end{align*}
	which implies that 
	\begin{align}\label{53}
		\l\{u_n^-\r\}_{n\in\N} \subseteq \Wpzero{p(\cdot)} \text{ is bounded}.
	\end{align}
	
	Now we choose $h=u_n^+ \in \Wpzero{p(\cdot)}$ as test function in \eqref{52}. This gives
	\begin{align}\label{54}
		-\varrho_{p(\cdot)}(\nabla u_n^+)-\varrho_{q(\cdot)}(\nabla u_n^+)+\into j_\lambda \l(x,u_n^+\r)u_n^+\,dx \leq \eps_n\quad \text{for all }n\in\N.
	\end{align}
	Furthermore, from \eqref{50} and \eqref{53}, we obtain
	\begin{align*}
		\l|\into \frac{1}{p(x)}|\nabla u_n^+|^{p(x)}\,dx +\into \frac{1}{q(x)} |\nabla u_n^+|^{q(x)}\,dx-\into J_\lambda \l(x,u_n^+\r)\,dx \r| \leq c_9
	\end{align*}
	for some $c_9>0$ and for all $n \in \N$. This implies
	\begin{align}\label{55}
		\varrho_{p(\cdot)}(\nabla u_n^+)+\varrho_{q(\cdot)}(\nabla u_n^+) -\into p_+ J_\lambda \l(x,u_n^+\r)\,dx \leq p_+c_9 \quad\text{for all }n\in\N.
	\end{align}
	We add \eqref{54} and \eqref{55} and obtain
	\begin{align*}
		\into \l[j_\lambda \l(x,u_n^+\r)u_n^+-p_+ J_\lambda \l(x,u_n^+\r)\r]\,dx \leq c_{10} \quad\text{for some }c_{10}>0 \text{ and for all }n\in\N,
	\end{align*}
	which by \eqref{41} results in
	\begin{align}\label{56}
		\into \lambda \l[f\l(x,u_n^+\r)u_n^+-p_+F\l(x,u_n^+\r)\r]\,dx \leq c_{11}\l(1+\into \l(u_n^+\r)^{1-\eta(x)}\,dx\r)
	\end{align}
	for some $c_{11}>0$ and for all $n \in \N$.
	
	Hypotheses H$_1$(i), (iii) imply the existence of $\gamma_1 \in \l(0,\gamma_0\r)$ and $c_{12}>0$ such that
	\begin{align}\label{57}
		\gamma_1s^{-\tau(x)}-c_{12} \leq f(x,s)s-p_+F(x,s)\quad\text{for a.\,a.\,}x\in \Omega \text{ and for all }s \geq 0.
	\end{align}
	Using \eqref{57} in \eqref{56}, we have
	\begin{align*}
		\varrho_{\tau(\cdot)} \l(u_n^+\r) \leq c_{13} \l[1+\into \l(u_n^+\r)^{1-\eta(x)}\,dx\r]
		\quad\text{for some }c_{13}>0 \text{ and for all }n\in\N. 
	\end{align*}
	Hence, we see that
	\begin{align}\label{58}
		\l\{u_n^+\r\}_{n\in \N} \subseteq \Lp{\tau(\cdot)} \text{ is bounded}.
	\end{align}
	
	From hypothesis H$_1$(iii) we see that, without any loss of generality, we may assume that $\tau(x)<r<p_-^*$ for all $x\in\close$. Hence, $\tau_-<r<p_-^*$ and so we can find $t \in (0,1)$ such that
	\begin{align}\label{59}
		\frac{1}{r}=\frac{1-t}{\tau_-}+\frac{t}{p_-^*}.
	\end{align}
	Applying the interpolation inequality, see Papageorgiou-Winkert \cite[p.\,116]{16-Papageorgiou-Winkert-2018}, we have
	\begin{align*}
		\l\|u_n^+\r\|_r \leq \l\|u_n^+\r\|_{\tau_-}^{1-t} \l\|u_n^+\r\|^t_{p_-^*}.
	\end{align*}
	Thus, due to \eqref{58},
	\begin{align*}
		\l\|u_n^+\r\|_r^r \leq c_{14} \l\|u_n^+\r\|^{tr}_{p_-^*} \quad\text{for some }c_{14}>0 \text{ and for all }n\in\N.
	\end{align*}
	Then, by the Sobolev embedding theorem, we obtain
	\begin{align}\label{60}
		\l\|u_n^+\r\|_r^r \leq c_{15} \l\|u_n^+\r\|^{tr} \quad\text{for some }c_{15}>0 \text{ and for all }n\in\N.
	\end{align}
	We take $h=u_n^+\in \Wpzero{p(\cdot)}$ in \eqref{52} as test function and get	
	\begin{align*}
		\varrho_{p(\cdot)}(\nabla u_n^+)+\varrho_{q(\cdot)}(\nabla u_n^+) \leq \eps_n +\into j_\lambda \l(x,u_n^+\r)u_n^+\,dx \quad\text{for all }n\in\N,
	\end{align*}
	which by \eqref{41} and \eqref{60} gives 
	\begin{align}\label{61}
		\begin{split} 
			\varrho_{p(\cdot)}(\nabla u_n^+)+\varrho_{q(\cdot)}(\nabla u_n^+) 
			& \leq c_{16} \l[1+\into \lambda f\l(x,u_n^+\r)u_n^+\,dx\r]\\
			& \leq c_{17}\l[1+\lambda \l\|u_n^+ \r\|_r^r\r]\\
			& \leq c_{18}\l[1+\lambda \l\|u_n^+ \r\|^{tr}\r]
		\end{split} 
	\end{align}
	for some $c_{16}, c_{17}, c_{18}>0$ and for all $n\in\N$.
	
	From \eqref{59} we have
	\begin{align*}
		tr=\frac{p_-^*\l(r-\tau_-\r)}{p_-^*-\tau_-}<p_-.
	\end{align*}
	Therefore, from \eqref{61} and Proposition \ref{proposition_1} it follows that
	\begin{align*}
		\l\{u_n\r\}_{n\in \N} \subseteq \Wpzero{p(\cdot)} \text{ is bounded}.
	\end{align*}
	
	So, we may assume that 
	\begin{align}\label{62}
		u_n \weak u \quad\text{in }\Wpzero{p(\cdot)} 
		\quad\text{and}\quad
		u_n \to u \quad\text{in }\Lp{p(\cdot)}.
	\end{align}
	
	We choose $h= u_n-u \in \Wpzero{p(\cdot)}$ in \eqref{52}, pass to the limit as $n\to \infty$ and apply \eqref{62}. This yields 
	\begin{align*}
		\lim_{n\to\infty} \l [\l\lan A_{p(\cdot)}\l(u_n\r),u_n-u \r\ran+\l\lan A_{q(\cdot)}\l(u_n\r),u_n-u\r\ran\r]=0.
	\end{align*}
	Note that $A_{q(\cdot)}(\cdot)$ is monotone, so we have
	\begin{align*}
		\limsup_{n\to\infty} \l [\l\lan A_{p(\cdot)}\l(u_n\r),u_n-u \r\ran+\l\lan A_{q(\cdot)}\l(u\r),u_n-u\r\ran\r]\leq 0.
	\end{align*}
	Because of \eqref{62} we then derive 
	\begin{align*}
		\limsup_{n\to\infty} \l\lan A_{p(\cdot)}\l(u_n\r),u_n-u \r\ran \leq 0
	\end{align*}
	and so, by Proposition \ref{proposition_2},
	\begin{align*}
		u_n \to u \quad \text{in }\Wpzero{p(\cdot)}.
	\end{align*}
	This proves the Claim.
	
	Then, \eqref{48}, \eqref{49} and the Claim permit us the use of the mountain pass theorem. So we can find $\hat{u}\in\Wpzero{p(\cdot)}$ such that
	\begin{align*}
		\hat{u} \in K_{w_\lambda} \subseteq [\overline{u})\cap \interior,
	\end{align*}
	see \eqref{44}, and 
	\begin{align*}
		w_\lambda \l(u_0\r)<m_\lambda\leq w_\lambda \l(\hat{u}\r),
	\end{align*}
	see \eqref{48}. We conclude that $\hat{u} \in \interior$ is the second positive solution of \eqref{problem} for $\lambda \in \l(0,\lambda^*\r)$ and $\hat{u}\neq u_0$.
\end{proof}

It remains to decide whether the critical parameter value $\lambda^*>0$ is admissible.

\begin{proposition}
	If hypotheses H$_0$ and H$_1$ hold, then $\lambda^* \in \mathcal{L}$.
\end{proposition}

\begin{proof}
	Let $\{\lambda_n\}_{n\in\N} \subseteq (0,\lambda^*)\subseteq \mathcal{L}$ be such that $\lambda_n\nearrow \lambda^*$ as $n \to \infty$. From the proof of Proposition \ref{11} we know that we can find $u_n \in \mathcal{S}_{\lambda_n} \subseteq \interior$ such that
	\begin{align*}
		w_{\lambda_n}(u_n) \leq w_{\lambda_n}(\overline{u}) \quad\text{for all }n\in\N.
	\end{align*}
	Applying \eqref{41}, $f \geq 0$ and Proposition \ref{proposition_4} we obtain
	\begin{align}\label{63}
		\begin{split} 
			&w_{\lambda_n}(u_n)\\
			& \leq \frac{1}{q_-} \l[\varrho_{p(\cdot)}\l(\nabla \overline{u}\r)+\varrho_{q(\cdot)}\l(\nabla \overline{u}\r)-\into \overline{u}^{1-\eta(x)}\,dx-\into \lambda_n f\l(x,\overline{u}\r)\overline{u}\,dx \r]\\
			& \leq \frac{1}{q_-} \l[\varrho_{p(\cdot)}\l(\nabla \overline{u}\r)+\varrho_{q(\cdot)}\l(\nabla \overline{u}\r)\r]-\into \overline{u}^{1-\eta(x)}\,dx\\
			& \leq \l[\frac{1}{q_-}-1\r] \l(\varrho_{p(\cdot)}\l(\nabla \overline{u}\r)+\varrho_{q(\cdot)}\l(\nabla \overline{u}\r)\r)<0
		\end{split} 
	\end{align}
	for all $n \in \N$. Furthermore, we have
	\begin{align}\label{64}
		\l \lan A_{p(\cdot)}\l(u_n\r),h\r\ran+\l \lan A_{q(\cdot)}\l(u_n\r),h\r\ran=\into j_\lambda \l(x,u_n\r)h\,dx
	\end{align}
	for all $h \in \Wpzero{p(\cdot)}$ and for all $n \in \N$.
	
	Using \eqref{63} and \eqref{64} and reasoning as in the Claim in the proof of Proposition \ref{proposition_11}, we obtain
	\begin{align*}
		u_n \to u^* \quad \text{in }\Wpzero{p(\cdot)} \quad\text{and}\quad \overline{u} \leq u^*,
	\end{align*}
	see Proposition \ref{proposition_6}. Hence, $u^* \in \mathcal{S}_{\lambda^*} \subseteq \interior$ and so $\lambda^* \in \mathcal{L}$.
\end{proof}

So, we have proved that
\begin{align*}
	\mathcal{L}=\l(0,\lambda^*\r].
\end{align*}

Summarizing our results we can state the following bifurcation-type result describing the changes in the set of positive solutions as the parameter moves on $\overset{\circ}{\R}_+=(0,+\infty)$.

\begin{theorem}
	If hypotheses H$_0$ and H$_1$ hold, then there exists $\lambda^*>0$ such that
	\begin{enumerate}
		\item[(a)]
			for every $\lambda\in (0,\lambda^*)$, problem \eqref{problem} has at least two positive solutions
		\begin{align*}
		u_0, \hat{u} \in \interior, \quad u_0\neq \hat{u};
		\end{align*}
		\item[(b)]
		for $\lambda=\lambda^*$, problem \eqref{problem} has at least one positive solution
		\begin{align*}
		u^*\in\interior;
		\end{align*}
		\item[(c)]
		for every $\lambda>\lambda^*$, problem \eqref{problem} has no positive solutions.
	\end{enumerate}
\end{theorem}


\end{document}